\documentclass[11pt]{amsart}
\usepackage{amsfonts, amsmath, amssymb, amscd, amsthm, graphicx, setspace, enumitem, color, xypic}
\usepackage{hyperref}
\newtheorem{thm}{Theorem}[section]
\newtheorem{cor}[thm]{Corollary}
\newtheorem{lem}[thm]{Lemma}
\newtheorem{prop}[thm]{Proposition}

\theoremstyle{definition}
\newtheorem{defn}[thm]{Definition}

\theoremstyle{remark}
\newtheorem{rem}[thm]{Remark}

\newcommand{\Z }{\mathbb Z}

\newcommand{\ra }{\rightarrow}

\begin{document}

\title{Torsion homology and cellular approximation}

\author{Ram\'on Flores and Fernando Muro}
\email[Ram\'on Flores]{ramonjflores@us.es}
\email[Fernando Muro]{fmuro@us.es}
\address{IMUS-Universidad de Sevilla, calle Tarfia, 4112, Sevilla (Spain)}

\thanks{The first author was partially supported by FEDER-MEC grant MTM2010-20692. The second author was partially supported by the Andalusian Ministry of Economy, Innovation and Science under the grant FQM-5713, and by FEDER-MEC grant MTM2010-15831. Both authors are partially supported by FEDER-MEC grant MTM2016-76453-C2-1-P.}

\maketitle

\begin{abstract}

In this note we describe the role of the Schur multiplier in the structure of the $p$-torsion of discrete groups. More concretely, we show how the knowledge of $H_2G$ allows to approximate many groups by colimits of copies of $p$-groups. Our examples include interesting families of non-commutative infinite groups, including Burnside groups, certain solvable groups and branch groups. We also provide a counterexample for a conjecture of E. Farjoun.

\end{abstract}



\section{Introduction}

Since its introduction by Issai Schur in 1904 in the study of projective representations, the Schur multiplier $H_2(G,\mathbb{Z})$ has become one of the main invariants in the context of Group Theory. Its importance is apparent from the fact that its elements classify all the central extensions of the group, but also from his relation with other operations in the group (tensor product, exterior square, derived subgroup) or its role as the Baer invariant of the group with respect to the variety of abelian groups. This last property, in particular, gives rise to the well-known Hopf formula, which computes the multiplier using as input a presentation of the group, and hence has allowed to use it in an effective way in the field of computational group theory. A brief and concise account to the general concept can be found in \cite{Ro96}, while a thorough treatment is the monograph by Karpilovsky \cite{Ka87}.

In the last years, the close relationship between the Schur multiplier and the theory of co-localizations of groups, and more precisely with the cellular covers of a group, has been remarked. The cellular notions were developed in the nineties by Farjoun and Chach\'olski (\cite{Fa96},\cite{Ch96}) in the category of pointed topological spaces, with a twofold goal: to define a hierarchy in the spaces by means of cellular classes, as Bousfield attempted with his periodicity functors \cite{Bo94} and Ravenel with thick categories in the stable category \cite{Ra92}, and to isolate in a rigorous way the structure of a space that depended of a fixed one, generalizing the idea of CW-complex by J.H.C. Whitehead. These ideas have shown to be very fruitful, and have been recently applied in different categories, as for example triangulated categories, $R$-modules, or derived categories (see for example \cite{DGI06}).

After developing the cellular theory for spaces, it was natural that the next context to apply these tools should be Group Theory: first, because some of the main classical invariants of Homotopy Theory are indeed groups (homology groups, homotopy groups, etc.); second, because there are important and well-known functors that pass from spaces to groups, in particular the fundamental group functor (in this sense, the relation between this functor and localization has already been studied, see \cite{Ca00}); and third, because the cellular theory is based on the notion of (homotopy) direct limit, which is in fact a generalization to the homotopy context of the direct limit of groups.

Given a group $G$, there are strong ties between the cellular covers of $G$ and its Schur multiplier. Recall that a \emph{cellular cover} of a group $G$ is a pair $(H,cw)$ where $H$ is a group and $cw:H\ra G$ is a map such that the induced map $Hom(H,H)\ra Hom(H,G)$ is a bijection. Given a pair of groups $A$ and $H$, there exists a distinguished cellular cover $(cell_AH,cw_A)$ which is roughly speaking the more accurate approximation of $H$ that can be constructed as direct limit (maybe iterated) of copies of $A$ (more information in next section). Since pioneering work of Rodr\'{i}guez-Scherer in \cite{RoSc01} and subsequent research by several authors, it has been established that the cellular cover $cell_AG$ of a group $G$ always lies in the middle of a central extension $K\ra cell_AG\ra S_AG$, where $S_AG$ is the subgroup of $G$ generated by images of homomorphisms from $A$. The kernel $K$ is usually the most difficult part to compute, and it is strongly related with the structure of the Schur multiplier of $G$ or a distinguished subgroup of it. In the next section we review the previous work in which the structure of $K$ has been investigated, particularly when $G$ is abelian.

Fix a prime $p$. In this paper we are interested in $\mathbb{Z}/p$-cellular approximation, which has given a new way to understand the $p$-primary structure of the group $G$. In this sense, we have been able to identify a wide range of groups for which the aforementioned kernel $K$ is a concrete quotient of the Schur multiplier of $S_{\Z /p}G$. For example, this will happen for groups for which $S_{\Z /p}G$ is finitely presented -or more generally finitely $L$-presented- groups for which the Schur multiplier of $S_p G$ is finitely generated, groups such that $H_2 S_pG$ is free abelian, or groups for which $H_2 S_pG$ is torsion. This study is undertaken in Section 3, and is primarily based on ideas of Chach\'olski that are described at the beginning of the section.

As we are describing a certain approach to the torsion in discrete groups, it is very interesting for us to understand the $\Z /p$-cellularization of torsion groups, a task which is carried out in Section 4. According to the previous paragraph, we are able to identify the kernel of the cellular cover when $H_2 S_pG$ is a torsion group, so we should know when this happens for a torsion group. In fact, as $S_pG$ is torsion if $G$ is, this leads us to a more general  question: if $G$ is a torsion group, when are so their ordinary homology groups? Work of Ol'shanskii and others guarantees that the general answer is no, but the counterexamples are quite exotic, and then it is interesting to identify ``big" classes of groups for which the answer to the question is positive. In this sense, we state in Proposition \ref{EAtorsion} that the ordinary homology is torsion for elementary amenable torsion groups. Moreover, we discuss some bigger classes for which this fact may hold, and compute the $\Z /p$-cellular approximation of some non-elementary amenable torsion groups, as the first Grigorchuk group and the Gupta-Sidki group. Continuing this study, in Section 5 we deal with the Burnside group $B(2,p)$ (for $p>665$), which is a torsion non-amenable group whose Schur multiplier is free abelian. We compute its $\Z /p$-cellularization (obtaining in this way a new counterexample to a conjecture of Farjoun) and we describe some exotic features of its classifying space.


\textbf{Acknowledgments.} We thank Primo\u{z} Moravec, Greg Oman, Adam Prze\'{z}dziecki, Andrei Jaikin, Laurent Bartholdi, Juan Gonz\'alez-Meneses and Sara Arias de Reyna for their useful comments and suggestions, Manfred Dugas for his example of a finite rank non-minimal abelian group, and J\'{e}r\^{o}me Scherer for his careful reading of the draft.

\section{Background}

In this section we will describe previous work relating the Schur multiplier and cellular covers.

As stated in the introduction, the first ideas that aimed to classify objects in a category by means of a cellular hierarchy go back to Farjoun and Chach\'olski. Inspired by early work of J.H.C. Whitehead, and working in the category of pointed spaces, they defined for every pair of spaces $A$ and $X$ the $A$-cellular approximation of $X$ as the ``closest" space to $X$ that can be built out of copies of $A$ by means of iterated (pointed) homotopy colimits. The notion of closeness here is given by the pointed mapping space (see Chapter 2 in \cite{Fa96} for a precise definition). Chach\'olski also defines the related notion of closed (or cellular) class of spaces.

Building on these ideas, Rodr\'{i}guez-Scherer define in \cite{RoSc01} the \emph{cellular class} $\mathcal{C}(A)$ of a group $A$ as the smallest class of groups that contains $A$ and is closed under colimits; its elements are called $A$-cellular groups. Then, they prove the existence of a functor $cell_A:Groups\rightarrow\mathcal{C}(A)$ that is right adjoint to the inclusion $\mathcal{C}(A)\rightarrow Groups$. The proof consists, for every group $G$, in building effectively $cell_AG$ as an infinite telescope of groups that are themselves free products of copies of $A$. This functor is augmented and idempotent, and for every $A$, $cell_AG$ will be called the $A$-\emph{cellularization} or $A$-\emph{cellular approximation} of $G$. The group $cell_AG$ can also be defined in the following way:

\begin{defn}

Given groups $A$ and $G$, the $A$-\emph{cellularization} of $G$ is the unique $A$-cellular group $cell_AG$ such that there exists a homomorphism $\eta:cell_AG\rightarrow G$ in such a way that any homomorphism $H\rightarrow G$ from a $A$-cellular group $H$ factors through $\eta$ in a unique way.

\end{defn}

Recall that given an abelian group $G$, a \emph{Moore space} is a space $M(G,n)$ such that $H_nM(G,n)=G$ and $H_jM(G,n)=0$ if $j\neq n$. When $A$ is a group such that there exists a Moore space $M(A,1)$ of dimension two (as for example, a finite cyclic group), Rodr\'{i}guez-Scherer also identify the $A$-cellularization of $G$ by means of a central extension: $$K\rightarrow cell_AG\rightarrow S_AG.$$ Here $S_AG$ is the $A$-\emph{socle} of $G$, i.e. the normal subgroup of $G$ generated by images of homomorphisms from $A$, while $K$ is a group such that $Hom(A_{ab},G)=0$. This kernel $K$ will be the key object that imbricate the cellular approximation with the homology of the base group. Moreover, the relation between cellularization of groups and cellularization of spaces with regard to Moore spaces is further investigated by the authors in \cite{RoSc08}.

The first attempt to understand the role of the homology in the cellular constructions was undertaken in \cite{Fl07}, for the case $A=\mathbb{Z}/p$, $p$ prime. As there is a 2-dimensional model for $M(\mathbb{Z}/p,1)$, the description of $cell_AG$ as an extension works, and from the homotopical construction of the extension it was deduced that, for $A$ finite, the kernel $K$ could be identified with the quotient of the Schur multiplier by its $p$-torsion group; in the present note we extend this result to infinite groups. Moreover, it was proved in that paper that in some cases in which $G$ is finite and perfect, $cell_{\mathbb{Z}/p}G$ can be identified with the universal central extension of $G$.

In \cite{FGS07}, Farjoun-G\"{obel}-Segev introduce a new point of view. For these authors, a homomorphism $A\rightarrow G$ is a \emph{cellular cover} if it induces by composition an isomorphism $Hom(A,A)\simeq Hom(A,G)$. It is not hard to see that $A\rightarrow G$ is a cellular cover if and only if $A$ is the $A$-cellular approximation of $G$. In particular, if one wishes to classify all the cellular covers of a given group $G$, it should be necessary to understand which groups can appear as kernels of the homomorphisms $A\rightarrow G$. In the mentioned paper it is proved for every cellular cover $A\rightarrow G$ that if $G$ is nilpotent then $K$ is torsion-free, if $G$ is abelian then $K$ is reduced, and if $G$ is finite then $K$ is so and $K\subseteq [A,A]$. It was also proved there that the universal central extension of any perfect group (whose kernel is the Schur multiplier) is always a cellular cover, generalizing the quoted result for finite groups.


After this paper, there has been a lot of interest in classifying cellular covers of abelian groups, as well as in describing the kernels of the covers. In the paper \cite{FGSS07}, the problem is solved for abelian divisible groups;  Fuchs-G\"{o}bel \cite{FuGo09}, relying in part on work of Buckner-Dugas \cite{BuDu06}, address the reduced case and give an accurate description of which groups can appear as kernels; Farjoun-G\"{obel}-Segev-Shelah \cite{FGSS07} have presented groups for which the cardinality of all the possible kernels of cellular covers is unbounded; Fuchs has investigated covers of totally ordered abelian groups \cite{Fu11}, and there is also work of Rodr\'{i}guez and Str\"{u}ngmann on the cotorsion-free case (\cite{RS12}, \cite{RS15}). However, the descriptions of the kernels given in these papers are usually not very explicit, and the authors do not investigate the possible relations of the kernels with the homology of the groups involved.

The relation of the cellular approximation with the second homology group is described in a more general way in \cite{CDFS09}. It is proved there that for finite groups $A$ and $G$ with $A$ finitely generated and such that $H_2A=S_GH_2A$, the kernel of the cellularization $cell_AG\rightarrow G$ is always the quotient of the Schur multiplier of $G$ by its $A$-socle, generalizing in this way the aforementioned result of \cite{Fl07}, that only dealt with the case $A=\mathbb{Z}/p$. This statement is true, in particular, when $A$ is finitely generated and nilpotent. It is also remarkable that this is proved using only group-theoretic tools, with no need of homotopical background. Some of the authors of the previous paper took a big step forward in \cite{BCFS12}, where it is proved (using the finiteness of the Schur multiplier for finite groups) that the number of cellular covers of finite group is always finite. If the group is moreover simple, then the non-trivial cellular covers of it are in bijective correspondence with the subgroups of the Schur multiplier that are invariant under the action of the automorphisms of the group.

To our knowledge, not a lot has been published concerning the cellular approximations of infinite non-nilpotent groups, so this note can be considered a step in that direction. We are only aware of G\"{o}bel's work \cite{Go12} and Petapirak's thesis \cite{Pe16} (see Section 5), although the work of these authors is addressed to a different problem, as it is the understanding of the group varieties that are closed under cellular covers and localization.

\section{Cellular covers}

We will start this section by fixing some notation and reviewing some previous results that we will need afterwards.

\subsection{Chach\'olski theory and notation}

The rigorous foundation of the theory of cellular approximations in the homotopy context was carried out by Chach\'olski in \cite{Ch96}. The following result was proved in that paper, and it is possibly the most powerful available gadget to compute $cell_A$. Recall that given a space $X$, its $A$-\emph{nullification} $P_A$ is its localization with respect to the constant map $A\rightarrow \ast$, and the spaces $X$ for which $P_AX$ is contractible are called $A$-\emph{acyclic}, see \cite{Fa96} for details. Analog definitions are used in the category of groups: given groups $A$ and $G$, the $A$-nullification $P_A$ is the localization of $G$ with respect to the constant homomorphism $A\rightarrow \{1\}$, and $G$ is $A$-acyclic if and only if $P_AG=\{1\}$.

\begin{thm}{\emph{(Theorem 20.5 in \cite{Ch96})}}

Let $A$ and $X$ be pointed spaces. Consider the evaluation map $f:\bigvee_{[A,X]_*}A\rightarrow X$, where the wedge is extended to all the homotopy classes of pointed maps $A\rightarrow X$. Let $C$ be the homotopy cofibre of the map $f$, and $\Sigma A$ the suspension of $A$. Then there exists a homotopy fibration: $$cell_AX\rightarrow X\rightarrow P_{\Sigma A}C$$

\end{thm}

In particular, given a group $G$, this result is crucial in proving that for certain choices of $A$, the group cellularization with respect to a group $A$ can be described by means of an extension $K\rightarrow cell_AG\rightarrow S_AG$, as explained in the previous section. As we are going to use this material frequently through the paper, the map $\bigvee_{[A,X]_*}A\rightarrow X$ will be called $\emph{Chach\'olski cofibration}$, its homotopy cofibre will be \emph{Chach\'olski cofibre}, and the fibre sequence $cell_AX\rightarrow X\rightarrow P_{\Sigma A}X$ will be referred to as \emph{Chach\'olski fibration}.

Moreover, given a discrete group $G$, we denote by $S_pG$ its $p$-socle, i.e. its subgroup $S_{\Z /p}G$ generated by order $p$ elements. Every group for which the inclusion $S_pG\leq G$ is an equality will be called $p$-generated. The previous description of $cell_{\Z /p}G$ as an extension immediately implies the following:

\textbf{Crucial fact}. For every discrete group $G$, the inclusion $S_pG\leq G$ induces an isomorphism $cell_{\Z /p}S_pG\simeq cell_{\Z /p}G$. In particular, the computation of $cell_{\Z /p}G$ is always reduced to the computation of $cell_{\Z /p}S_pG$, and it will be only necessary to compute $\Z /p$-cellular approximations of $p$-generated groups.

Let $p$ a fixed prime, $G$ a discrete $p$-generated group. We are interested in $\Z /p$-cellularization, and its construction is derived \cite{RoSc01} from cellularization with respect to Moore spaces, so we will denote by $M$ a two-dimensional model for the Moore space $M(\Z /p,1)$, as for example the homotopy cofibre of the degree $p$ self-map of the circle. This space will play the role of $A$ in Chach\'olski fibration, while $X$ will be $BG$, the classifying space of the group $G$. Once the group $G$ is fixed, Chach\'olski cofibration and fibration will be considered with respect to these choices of $A$ and $X$. Usually Chach\'olski cofibre will be denoted by $C$ if the group is understood. Moreover, in the remaining of the paper, the homology will be always considered with coefficients in $\Z$. Finally, given an abelian group $L$, we will denote by $TL$ its torsion subgroup, and by $T_pL$ its $p$-torsion subgroup.

We would like to remark that, \emph{mutatis mutandis}, the majority of our results will be also valid if we change $\Z /p$ by $\Z /p^j$ with $j>1$. Moreover, given a set $P$ of primes, they will remain correct changing $\mathbb{Z}/p$ by the free product of copies of the cyclic groups $\Z /q$, $q\in P$, and $M(\Z /p,1)$ by the corresponding Moore space with regard to the direct sum of the primes in $P$. However, for the sake of readiness, we will concentrate in the case of just one prime.

\subsection{The kernel of the cellular approximation}

In our search of a characterization of a sharper description of $cell_{\Z /p}G$, our initial goal was to generalize Theorem 4.4 in \cite{Fl07} for infinite discrete groups. That statement gives a way to compute the $\Z /p$-cellular approximation of a ($p$-generated) finite group out of a certain quotient of its Schur multiplier, but only for groups subject to quite strong restrictions. In this section we show how to relax these restrictions, and prove that the result is true far beyond the realm of finite groups. We start by giving a useful definition that will help to deal with the class of groups for which our methods will work.

\begin{defn}

We define the class $\mathcal{C}_p$ as the class of $p$-generated groups for which the kernel of the map $cell_{\Z /p}G\rightarrow G$ is isomorphic to $H_2G/T_pH_2G$.

\end{defn}

Our goal will be to identify groups whose $p$-socle lies in the class $\mathcal{C}_p$, as for these groups the kernel of the $\Z /p$-cellularization is then perfectly described. Our key result will be the following, because it gives a criterion to decide when a $p$-generated group belongs to $\mathcal{C}_p$.

\begin{prop}

In the previous notation, let $G$ be $p$-generated, and assume that there is an isomorphism $H_2G/T_pH_2G\simeq H_2C/T_2H_2C$. Then $G$ belongs to $\mathcal{C}_p$.

\label{H2iso}
\end{prop}
\begin{proof}

Consider Chach\'olski cofibration for the classifying space $BG$, and let $C$ be its homotopy cofibre. It is clear by Whitehead Theorem that $C$ is simply-connected, so the Hurewicz homomorphism $\pi_2C\rightarrow H_2C$ is an isomorphism. In particular, by Lemma 6.9 in \cite{Dw96}, $P_{\Sigma M}C$ is also simply-connected, and the definition of nullification as a limit of a sequence of push-outs implies that $\pi_2P_{\Sigma M}C=\pi_2C/T_p\pi_2C$, since $M$ is a Moore space. But as $\pi_2C=H_2C$ again by Hurewicz theorem, the hypothesis implies that $\pi_2P_{\Sigma M}C\simeq H_2G/T_pH_2G$, and now the result follows from the long exact homotopy sequence of Chach\'olski fibration.

\end{proof}

Although the nature of the previous result is quite technical, it opens the door to identify the kernel of $cell_{\Z /p}G$ for big families of groups. We will start our description of them with an easy lemma.

\begin{lem}
\label{p'-torsion}

Let $p$ be a prime, and $P$ its complementary sets of primes. Consider an extension $A\stackrel{i}{\rightarrow }B\rightarrow V$ of abelian groups, with $V$ a $\mathbb{F}_p$-vector space.  Let $T_{P}A$ and $T_{P}B$ the corresponding subgroups whose elements are $P$-torsion. Then a homomorphism $A\rightarrow B$ is an isomorphism if and only if the induced homomorphism $A/T_PA\rightarrow B/T_PB$ is an isomorphism.

\end{lem}

\begin{proof}

The 'only if' part is clear. For the 'if' part, observe that the structure of the previous extension implies that $T_PA$ is isomorphic to $T_PB$, and moreover every isomorphism between $A$ and $B$ restricts to an isomorphism between $T_PA$ and $T_PB$. Hence, the five lemma applied to the diagram

$$
\xymatrix{ T_PA \ar[r] \ar[d] & A  \ar[r] \ar[d] & A/T_PA \ar[d] \\
T_pB \ar[r] & B  \ar[r] & B/T_PB }
$$

concludes the argument.

\end{proof}

The previous result proves that the torsion outside $p$ in the Schur multiplier is not relevant when deciding the possible equality between $H_2G / H_2T_pG$ and $H_2C / H_2T_pC$. Hence, we can formulate the following statement:

\begin{prop}
\label{H2torsionfree}

Let $G$ be a $p$-generated group such that the torsion-free groups $H_2G/ TH_2G$ and $H_2C/TH_2C$ are isomorphic. Then $G\in \mathcal{C}_p$.

\end{prop}

\begin{proof}

Consider again Chach\'olski cofibration for $BG$. As $H_2M=0$, Mayer-Vietoris sequence gives an exact sequence of abelian groups:

$$0\rightarrow H_2 G\rightarrow H_2C\rightarrow V \rightarrow H_1 G\rightarrow 0,$$ where $V$ is an $\mathbb{F}_p$-vector space.


 As $G$ is $p$-generated, $H_1G$ is also an $\mathbb{F}_p$-vector space, and $H_2C$ can be described as an extension $ H_2G\rightarrow H_2C\rightarrow W$, where the dimension of the $\mathbb{F}_p$-vector space $W$ is smaller or equal to that of $V$ (in particular, they will be the same if and only if $G$ is perfect). By Proposition \ref{H2iso}, the statement holds if $H_2 G/T_pH_2G$ is isomorphic to $H_2 C/T_pH_2C$, and by Lemma \ref{p'-torsion} if $H_2 G/TH_2G$ is isomorphic to $H_2 C/TH_2C$. So we are done.

\end{proof}

From now on, we will concentrate in finding $p$-generated groups for which the hypothesis of the previous result hold. The first relevant and important case appears when the Schur multiplier is torsion.

\begin{prop}
\label{torsion}

Let $G$ be a $p$-generated group whose Schur multiplier is torsion. Then $G$ belongs to $\mathcal{C}_p$.

\end{prop}

\begin{proof}

Let $G$ be $p$-generated, and recall from the proof of Proposition \ref{H2iso} the extension $H_2G\rightarrow H_2C\rightarrow W$, which is derived from Mayer-Vietoris exact sequence of Chach\'olski cofibration. As $H_2G$ is torsion, $H_2C$ is so. Hence, the respective quotients by the torsion are trivial, and then isomorphic, so we can apply again Proposition \ref{H2iso} and we are done.

\end{proof}

\begin{rem}

Note that the isomorphism between the quotients by the $p$-torsion is in this case induced by the map $BG\rightarrow C$. This is in fact stronger than the hypothesis of Proposition \ref{H2iso}, in which it is not needed that the isomorphism is induced in that way.

\end{rem}

The condition about the Schur multiplier in the previous proposition does \emph{not} immediately imply that the previous proposition is true for torsion groups. In fact, to understand for which torsion groups it is true that their homology is torsion is an interesting issue that will be the main topic of the next section.

 Next we will deal with $p$-generated groups such that $H_2G$ is not torsion, and we will need to check when $H_2G/TH_2G$ is isomorphic to $H_2C/TH_2C$. These quotients are torsion-free, so we are led to the following interesting group-theoretic question:

\textbf{Question}. Let $A\stackrel{f}{\rightarrow }B\stackrel{g}{\rightarrow}V$ be an extension of abelian groups, where $V$ is an $\mathbb{F}_p$-vector space for a certain prime $p$, and $A$ and $B$ are $p$-torsion free. When are $A$ and $B$ isomorphic?

The general answer to this question is related with the notion of \emph{minimal abelian group}, i.e. abelian groups which are isomorphic to every finite index subgroup (see \cite{Oh05}). In Theorem 3.23 there it is described an abelian group of infinite rank for which the statement does not hold: the Baer-Specker group $\prod_{{\aleph}_0}\mathbb{Z}$, that for every prime $p$, has a subgroup of index $p$ which is not isomorphic to it. M. Dugas communicated to the authors a counterexample with finite rank: let $A$ and $B$ subgroups of $\mathbb{Q}$ such that $1/p$ does not belong to $A$ neither to $B$, and such that $Hom(A,B)=Hom(B,A)=0$. Then consider the subgroup of $\mathbb{Q}\oplus\mathbb{Q}$ given by $G=(A(1,0)\oplus B(0,1))+1/p(1,1)$. Then $G/(A(1,0)\oplus B(0,1))=\Z /p$, but $G$ and $A(1,0)\oplus B(0,1)$ are not isomorphic.

Hence, it will probably not be possible to always describe the kernel of the cellularization as a quotient of the Schur multiplier; in this sense, it is interesting to recall that every abelian group can be realized as the Schur multiplier of a $p$-generated perfect group (Theorem 1.1 in \cite{BeMa07}). However, the desired description of the kernel will be available in most usual cases, and this will be the subject of the rest of the section. Let us deal first with the free abelian case.

\begin{lem}

Let $A\rightarrow B\rightarrow C$ an extension of abelian groups, with $B$ free and $C$ torsion. Then $A$ and $B$ are isomorphic.

\end{lem}

\begin{proof}

As $A$ is a subgroup of an free abelian group, it is itself free abelian. Moreover, the rank of abelian groups is additive in exact sequences. Then, as the rank of a torsion group is zero, $\textrm{rk }A=\textrm{rk }B$. But two free abelian groups are isomorphic if and only if they have the same rank, so we are done.

\end{proof}

Now we have the following:

\begin{prop}
\label{decom}
Let $G$ be a $p$-generated group such that $H_2G=F\oplus T$, being $F$ free abelian and $T$ torsion; then $G\in \mathcal{C}_p$.

\end{prop}

\begin{proof}

According to Lemma \ref{p'-torsion}, $H_2G/T_pH_2G$ is isomorphic to $H_2C/T_pH_2C$ if and only if $H_2G/TH_2G$ is isomorphic to $H_2C/TH_2C$, and these two groups are free abelian in this case. Again by the Mayer-Vietoris exact sequence of Chach\'olski cofibration, we have an extension $F\rightarrow H_2C/TH_2C\rightarrow V$, with $V$ and $\mathbb{F}_p$-vector space. Hence, by Proposition \ref{H2iso} we only need to prove that $F$ and $H_2C/TH_2C$ are isomorphic.
For brevity, we denote $L=H_2C/TH_2C$ here.

Let us check first that $L$ is free abelian. Observe that for every element $x$ of $L$, $px$ belongs to $F$. Hence, $pL\subseteq F$, and it is free abelian, because it is a subgroup of a free abelian group. But the multiplication by $p$ is an isomorphism over the image for torsion-free abelian groups, so $L\simeq pL$ and hence it is free abelian. Now by additivity of the rank, $\textrm{rk }F=\textrm{rk }L$, and we are done.
\end{proof}

\begin{cor}
\label{fg}
If $G$ is $p$-generated and $H_2G$ is finitely generated, then $G\in\mathcal{C}_p$.
\end{cor}






 Observe that if $G$ is finitely presented, $H_2G$ is finitely generated, and we are under the hypothesis of the previous corollary. In fact, we are able to identify the kernel of the cellular approximation for a far more general class of groups that was introduced by Bartholdi in \cite{Ba03}:

 \begin{defn}

 Let $S$ be an alphabet, $F_S$ the free group in $S$ and $\Phi$ a set of endomomorphisms of $F_S$. An $L$-\emph{presentation} is an expression $\langle S|Q|\Phi|R\rangle$, where $Q$ and $R$ are sets of reduced words in $F_S$. When $S$, $Q$, $\Phi$ and $R$ are finite sets the $L$-presentation is said $finite$. A group $G$ is $L$-presented if there exists an $L$-presentation $\langle S|Q|\Phi|R\rangle$ such that $$G=F_S/\langle Q\cup\bigcup_{\phi\in\Phi^*}\phi(R) \rangle^{\#}. $$ Here $\#$ denotes normal closure and $\Phi^*$ is the monoid generated by $\Phi$.
  \end{defn}

As stated in Proposition 2.6 of \cite{Ba03}, the class of finitely $L$-presented groups contains strictly the family of finitely presented groups, as well as for example free Burnside groups, free solvable groups of finite rank and some instances of branch groups. The following result proves that for these groups is possible to compute explicitly the kernel of $cell_{\Z /p}G\rightarrow S_pG$:

\begin{prop}

Every $p$-generated finitely $L$-presented group belongs to $\mathcal{C}_p$.

\end{prop}

\begin{proof}

In Theorem 2.16 of \cite{Ba03}, it is proved that the Schur multiplier of a finitely $L$-presented can always be decomposed as the direct sum of a torsion group and a free abelian group. The statement is now a direct consequence of Proposition \ref{decom}.

\end{proof}

This result will be used in the following sections to compute explicitly $\Z /p$-cellular approximations of infinite torsion groups.

\begin{rem}

When computing the $\Z /p$-cellular approximation of a group, the usual strategy is to compute first the $p$-socle, and then compute the $\Z /p$-cellularization of the socle, which is the same as the one of the original group. The process must be carried out in this order, as we are not aware of general relations between $H_2G$ and $H_2 S_pG$ that are relevant in our context. This seems to be a good topic to perform ulterior research.

\end{rem}

\section{Homology of torsion groups}

In this section we will identify some big classes of torsion groups whose homology is torsion,  that are then appropriate candidates for computing $cell_{\Z /p}$ using the exact sequence of Proposition \ref{H2iso}. We found it surprising that it seems not to be a general treatment of this problem in the literature.

From now on, let $P$ be a non-empty set of primes. For us, a $P$-torsion group is a group $G$ such that for every $x\in G$, there exists $p\in P$ such that $x$ is $p$-torsion.

\begin{prop}
\label{locallyfinite}
Let $G$ be a locally finite $P$-torsion group. Then its homology groups are $P$-torsion.

\end{prop}

\begin{proof}

First, we know that every finite $P$-torsion group has $P$-torsion homology. Then, we use that every group is a filtered colimit of its finitely generated subgroups, and we find that $G$ is a filtered colimit of its finite subgroups. Now, as the homology commutes with filtered colimits and the colimit of a filtered system of $P$-torsion groups is $P$-torsion, the homology of $G$ is $P$-torsion.

\end{proof}


This result opens the door to prove the same for a very interesting class of groups, which was defined by Day \cite{Da57}.

\begin{defn}
The class of \emph{elementary amenable} groups is the smallest class of groups that contains abelian and finite groups and is closed under isomorphisms, subgroups, quotients, extensions and directed unions.
\end{defn}

Chou proved in Section 2 of \cite{Ch80} that the conditions of Day definition concerning subgroups and quotients are redundant. More precisely, and following his notation, let $EG_0$ be the class whose elements are abelian groups and finite groups. If $\alpha$ is a successor ordinal, the groups in $EG_{\alpha}$ are obtained by performing one extension of two elements of $EG_{\alpha-1}$, or a directed union over elements of $EG_{\alpha-1}$. Moreover, if $\alpha$ is a limit ordinal, then $EG_{\alpha}=\bigcup_{\beta<\alpha}EG_{\beta}$. In this way, the class of elementary amenable groups is defined as the union $EG=\bigcup EG_{\alpha}$ for all ordinals $\alpha$. Note in particular that every soluble group is elementary amenable.

Now we can state our main result of this section:

\begin{prop}
\label{EAtorsion}

Let $G$ be a elementary amenable $P$-torsion group. Then its homology groups are $P$-torsion.

\end{prop}

\begin{proof}  It is enough to check that every torsion elementary amenable group is locally finite, and this is true by Theorem 2.3 in \cite{Ch80}. So we are done.

\end{proof}

This result allows to describe the $\Z /p$-cellular structure of the elementary amenable torsion groups:

\begin{cor}
\label{CellEA}

If $G$ is a $p$-generated elementary amenable $p$-torsion group for some prime $p$, then $G\in\mathcal{C}_p$. In particular, every $p$-torsion and $p$-generated elementary amenable group is $\Z /p$-cellular.

\end{cor}

\begin{proof}

As $G$ is torsion, its homology groups are so by the previous proposition. Then by Proposition \ref{torsion}, $G$ belongs to $\mathcal{C}_p$. Moreover, if $G$ is $p$-torsion, $H_2G$ is so, and then the kernel of the augmentation $cell_{\Z /p}G\rightarrow G$, which is isomorphic to $H_2G/T_pH_2G$, is trivial. Thus, $G$ is $\Z /p$-cellular.
\end{proof}

\begin{rem}

The proof of Proposition \ref{EAtorsion} we provide, based on local finiteness, is short and concise. It is also possible to prove the previous proposition combining Chou's description of the elementary amenable groups with a clever use of the nullification with regard to $P_M$, or else appealing to Serre class theory.


\end{rem}

Recall that given a class of groups $\mathcal{F}$, the \emph{acyclic class} $\bar{\mathcal{C}}(\mathcal{F})$ generated by $\mathcal{F}$ is the smallest class that contains $\mathcal{F}$ and is closed under colimits and extensions. Let $EA$ be the class of elementary amenable $P$-torsion groups, $LF$ the class of locally finite $P$-torsion groups, and $PG$ the class whose elements are the cyclic groups $C_p$, for $p\in P$.

\begin{prop}

There are bijections $\bar{\mathcal{C}}(PG)\simeq \bar{\mathcal{C}}(EA)\simeq \bar{\mathcal{C}}(\mathcal{LF})$. In particular, every push-out of $P$-torsion elementary amenable and/or locally finite groups has $P$-torsion homology.

\end{prop}

\begin{proof}

As $PG\subset EA\subset LF$, so $\bar{\mathcal{C}}(PG)\subset \bar{\mathcal{C}}(LF)$. We have to check the other inequality.  Let $M_P$ be a wedge $\bigvee_{p\in P}M(\Z /p,1)$ of two-dimensional Moore spaces. If $G$ is a $P$-torsion finite group, its $M_P$-nullification is trivial. Hence, $G$ belongs to $PG$. Now every locally finite $P$-torsion group is a colimit of finite $P$-torsion groups. As the acyclic classes are closed under colimits, we have $\bar{\mathcal{C}}(LF)\subset \bar{\mathcal{C}}(PG)$ and we are done.

On the other hand, we have proved that the homology of every group in $EA$ is $P$-torsion. Using the same argument of Theorem 6.2 in \cite{RoSc01}, it is clear that the classifying space of every $P$-torsion group in $EA$ is $M_P$-acyclic, and this implies that the group belongs to $\bar{\mathcal{C}}(\mathcal{P})$, and in particular has $\mathcal{P}$-torsion homology.


\end{proof}

\begin{rem} Observe in particular that the push-out operation produces examples of groups whose homology is $P$-torsion but they are not $P$-torsion themselves, as for example the free product $C_p\ast C_p$ for $p\in P$. So, the closeness of these classes under arbitrary push-outs gives an easy way to produce groups with $P$-torsion homology which are not $P$-torsion themselves.
\end{rem}


Let us now concentrate in the case in which $P$ has only one prime $p$. So far, all the classifying spaces of the $p$-torsion groups we have dealt with in this section are $M$-acyclic. In particular, this implies that the homology of these groups is $p$-torsion, and in particular the groups are $\mathbb{Z}/p$-cellular. Now we will review some examples of $p$-torsion groups whose classifying spaces are \emph{not} $M$-acyclic, and hence not $M$-cellular. These examples are not amenable, so the next question seems really interesting:

\textbf{Question:} Is the classifying space of every amenable $p$-torsion group $M$-acyclic?

The answer to this question can be quite difficult, because so far no constructive scheme is available to obtain the class of amenable groups out of the class of elementary amenable groups. It is also remarkable that, unlike what happens in the elementary amenable case, there exist finitely generated $p$-torsion amenable groups that are infinite. Among these ones, the most famous examples are probably the first Grigorchuk group \cite{Gr84} and the Gupta-Sidki 3-group \cite{GuSi83}. Our methods allow to compute their relevant cellular approximations:

\begin{prop}

The first Grigorchuk group $\mathfrak{G}$ is $\Z /2$-cellular, and the $3$-torsion Gupta-Sidki group $\bar{\bar{\Gamma}}$ is $\Z /3$-cellular.

\end{prop}

\begin{proof}

First, recall that $\mathfrak{G}$ is $2$-generated and $\bar{\bar{\Gamma}}$ is 3-generated, so we must check that the kernels of the corresponding augmentations $cell_{\Z /2}\mathfrak{G}\rightarrow \mathfrak{G}$ and $cell_{\Z /3}\bar{\bar{\Gamma}}\rightarrow\bar{\bar{\Gamma}}$ are trivial. According to Section 4 in \cite{Ba03}, the Schur multipliers of these groups are $\mathbb{F}_p$-vector spaces on a countable number of generators, with $p=2$ for $\mathfrak{G}$ and $p=3$ for $\bar{\bar{\Gamma}}$ (the original computation for $\mathfrak{G}$ can be found in \cite{Gr99}). Then by Proposition \ref{decom}, $\mathfrak{G}\in\mathcal{C}_2$ and $\bar{\bar{\Gamma}}\in\mathcal{C}_3$. Hence, as $H_2\mathfrak{G}$ and $H_2\bar{\bar{\Gamma}}$ are respectively 2-torsion and 3-torsion, the aforementioned kernels are trivial, and the groups are cellular.

\end{proof}

Observe that the result implies that these groups can be constructed out of the corresponding $\mathbb{Z}/p$ by means of iterated telescopes and push-outs. However, as they are not elementary amenable, they cannot be constructed out of these finite cyclic groups by extensions and directed unions. This proves in particular that the push-out operation marks a difference between amenable and elementary amenable groups.

On the other hand, we do not know any example of any amenable non-elementary amenable $p$-torsion group such that its Schur multiplier is not $p$-torsion. The examples of this class for which the Schur multiplier is known are residually nilpotent, and the Schur multiplier is $p$-torsion (although not in general finitely generated). It would be interesting to know if every residually nilpotent infinite $p$-group has a $p$-torsion Schur multiplier.

\section{Torsion in the homotopy}

In \cite{Ch08}, Emmanuel Farjoun proposes an interesting list of conjectures that concern localizations and cellular approximation. In this section we exhibit counterexamples of Conjecture 8 of the list, that asks if given a space whose homotopy groups are $p$-torsion, the same condition holds for their localizations or cellularizations. We notice that a counterexample for the second part based on a Tarski monster and a different cellularization can be extracted from recent work of G\"{o}bel \cite{Go12} and Petapirak \cite{Pe16} about varieties of groups.

It must be remarked that it is not complicated to find examples of idempotent functors and classes of spaces for which Farjoun's statement holds. This is clear for the $n$-connected cover functor or the Postnikov towers, for instance. If we look for more sophisticated and mod $p$ meaningful examples, we have that if $X$ is any simply-connected space whose homotopy groups are finitely generated $p$-groups, same holds for the homological localization $H\mathbb{Z}/p (X)$  \cite{Bo75}, and also for the nullification with respect to any Moore space $M(\mathbb{Z}/p,n)$ for every $n\geq 1$ \cite{Bo97}. In the augmented case, the statement holds again for any $X$ and the functor $cell_{M(\mathbb{Z}/p,n)}$ \cite{Ch96}, and for $cell_{B\mathbb{Z}/p}$ and $X=K(P,n)$, with $n\geq 1$ and $P$ a nilpotent $p$-torsion group \cite{CFFS15}.

A counterexample for the coaugmented case can be deduced from Bousfield's work.

\begin{prop}

If $X$ is the classifying space of the Pr\"{u}fer group $\mathbb{Z}/p^{\infty}$, the $H\mathbb{Z}/p$-homological localization of $X$ has the homotopy type of $K(\mathbb{Z}^{\wedge}_p,2)$.

\end{prop}

\begin{proof}

It is known that the $H\mathbb{Z}/p$-localization coincides con Bousfield-Kan $p$-completion over $p$-good spaces, so in this case is enough to compute $X^{\wedge}_p$. The space $X$ is $p$-good, because it is the classifying space of an abelian group. Then, we can use the exact sequence of \cite{BK72}, VI.5.1. As $Hom(\mathbb{Z}/p^{\infty},\mathbb{Z}/p^{\infty})=\mathbb{Z}^{\wedge}_p$ the result follows.

\end{proof}

Now let $B(2,p)$ the free Burnside group in two generators, for $p>665$. This is a group of exponent $p$, and in fact it is the free group in the variety of groups of exponent $p$. It is a famous result by Ol'shanskii (Corollary 31.2 in \cite{Ol91}) that the Schur multiplier of this group is a free abelian group in a countable number of generators. This is the key result for our following counterexample. Recall that $M$ is a two dimensional Moore space for $\Z /p$.

\begin{prop}

Let $G$ be $B(2,p)$, for $p>665$. Then the fundamental group of $cell_MBG$ is not $p$-torsion.

\end{prop}

\begin{proof}

 As $G$ has exponent $p$, it is of course equal to its $p$-socle. As $G$ is finitely $L$-generated (Proposition 2.14 in \cite{Ba03}), $G\in\mathcal{C}_p$, and hence its $\mathbb{Z}/p$-cellularization is defined by an extension

 $$F\rightarrow cell_{\Z /p}G\rightarrow G,$$ where $F=H_2G$ is free abelian and countable, and in particular not $p$-torsion. We conclude by recalling from Theorem 2.7 in \cite{RoSc01} that $\pi_1cell_{M}BG=cell_{\Z /p}G$.

\end{proof}

Note that this fundamental group is in particular $p$-generated.

The examples of exotic $p$-groups that appeared around Burnside problem have some additional features from the point of view of Homotopy Theory. In particular, they supply interesting information about how limited are the localization or completion functors when these are used to identify $p$-torsion in classifying spaces. For example, $B(2,p)$ is $p$-torsion but its classifying space is not $M$-cellular:

\begin{prop}

The classifying space of $B(2,p)$ is not $M$-acyclic (and in particular not $M$-cellular) if $p>665$.

\end{prop}

\begin{proof}

As we have already seen, the Schur multiplier of $B(2,p)$ is torsion-free. But according to Theorem 6.1 in \cite{RoSc01}, the homology of an $M$-acyclic space should be $p$-torsion. Hence, $P_MK(B(2,p),1)$ is non-trivial.

\end{proof}

There is an immediate corollary concerning the position of the classifying space of this Burnside group in the $B\mathbb{Z} /p$-cellular and acyclic hierarchies, which have deserved some interest in the last years (see for example \cite{Dw96}, \cite{CCS07} or \cite{FF11}).

\begin{prop}

For $p>665$, the classifying space of $B(2,p)$ is not $B\mathbb{Z}/p$-acyclic, and hence not $B\mathbb{Z}/p$-cellular.

\end{prop}

\begin{proof}

It is enough to recall that $B\mathbb{Z}/p$ is $M$-cellular and $M$-acyclic.

\end{proof}

It is an interesting and probably very difficult problem to find the concrete values of these functors over the classifying space of $B(2,p)$. Same can be said about Bousfield-Kan completion (see \cite{BK72}) of this space, but at least we can say the following:

\begin{prop}

Let $p>665$ be a prime number. Then for another prime $q\neq p$, the Bousfield-Kan $q$-completion of the classifying space of $B(2,p)$ is a simply connected non-contractible space.

\end{prop}

\begin{proof}

As $B(2,p)$ is $p$-torsion, it is $q$-perfect, and hence $K(B(2,p),1)^{\wedge}_q$ is simply-connected by \cite{BK72}, VII.3.2. Moreover, as the Schur multiplier is infinite cyclic, the $\mathbb{F}_q$-homology of the classifying space is non-trivial. As it is also $q$-good (because it is $q$-perfect), the $\mathbb{F}_q$-homology is preserved by $q$-completion, so $K(B(2,p),1)^{\wedge}_q$ is not contractible.

\end{proof}

Hence, this Burnside group gives an example of a $p$-group such that the $q$-completion of its classifying space is nontrivial.

\end{document}